\documentclass[oneside,a4paper]{amsart}

\usepackage{graphicx}
\usepackage{amsmath}
\usepackage{amssymb}
\usepackage{xspace}
\usepackage{subfigure}
\usepackage{pstricks, pst-node}
\usepackage[width=15.2cm, height=22cm]{geometry}
\def\eg{\emph{e.g.,}\xspace}
\def\ie{\emph{i.e.,}\xspace}
\def\resp{\emph{resp.}\xspace}
\def\wrt{\emph{w.r.t.}\xspace}
\def\A{\ensuremath{\mathcal{A}}\xspace}
\def\B{\ensuremath{\mathcal{B}}\xspace}
\def\cP{\mathcal{P}}
\def\Tb{T_{\mathrm{blue}}}

\def\ua{\underline{a}}
\def\ub{\underline{b}}
\def\uc{\underline{c}}
\def\ul{\underline{\ell}}

\def\Wt{W_{\mathrm{t}}}
\def\Wm{W_{\mathrm{m}}}
\newpsobject{mygrid}{psgrid}{gridcolor=lightgray,subgriddiv=1,gridlabels=0pt}
\psset{unit=.65em}
\newcommand{\cacher}[1]{}
\def\theta{\vartheta}

\newtheorem{property}{Property}

\newtheorem{theorem}{Theorem}
\newtheorem{proposition}[theorem]{Proposition}
\newtheorem{lemma}[theorem]{Lemma}
\newtheorem{claim}[theorem]{Claim}

\author[\'E. Fusy, G. Schaeffer]{\'Eric Fusy}
\address{\'E. Fusy, G. Schaeffer: LIX, \'Ecole Polytechnique, 91128 Palaiseau Cedex, France.}
\email{fusy@lix.polytechnique.fr, schaeffe@lix.polytechnique.fr}
\author[D. Poulalhon]{Dominique Poulalhon}
\address{D. Poulalhon: LIAFA, Universit\'e Paris Diderot,  case 7014, 75205 Paris Cedex 13, France.}  
\email{Dominique.Poulalhon@liafa.jussieu.fr}

\author{Gilles Schaeffer}
\title[Bijective counting of plane bipolar orientations and Schnyder woods]{Bijective counting of plane bipolar orientations\\ and Schnyder woods}
\begin{document}
%%%%%%%%%%%%%%%%%%%%%%%%%%%%%%%%%%%%%%%%%%%%%%%%%%%%%%%%%%%%%%%%%%%%%%%%%%%%
%%%%%%%%%%%%%%%%%%%%%%%%%%%%%%%%%%%%%%%%%%%%%%%%%%%%%%%%%%%%%%%%%%%%%%%%%%%%

\begin{abstract}
  A bijection $\Phi$ is presented between plane bipolar orientations
  with prescribed numbers of vertices and faces, and non-intersecting
  triples of upright lattice paths with prescribed extremities.
  This yields a combinatorial proof of the following formula due to
  R.~Baxter for the number $\Theta_{ij}$ of plane bipolar orientations with
  $i$ non-polar vertices and $j$ inner faces:
  \[
    \Theta_{ij} ~=~ 2 ~
    \frac{(i+j)!~(i+j+1)!~(i+j+2)!}{i!\;(i+1)!\;(i+2)! ~
    j!\;(j+1)!\;(j+2)!}.
  \]
  In addition, it is shown that $\Phi$ specializes into the bijection
  of Bernardi and Bonichon between Schnyder woods and non-crossing
  pairs of Dyck words.
  
  \emph{This is the extended and revised journal version of a
    conference paper with the title ``Bijective counting of plane
    bipolar orientations'', which appeared in Electr. Notes in
    Discr. Math. pp. 283-287 (proceedings of Eurocomb'07, 11-15
    September 2007, Sevilla).}
\end{abstract}

\maketitle
 
%%%%%%%%%%%%%%%%%%%%%%%%%%%%%%%%%%%%%%%%%%%%%%%%%%%%%%%%%%%%%%%%%%%%%%%%%%%%
\section{Introduction}
%%%%%%%%%%%%%%%%%%%%%%%%%%%%%%%%%%%%%%%%%%%%%%%%%%%%%%%%%%%%%%%%%%%%%%%%%%%%

A \emph{bipolar orientation} of a graph is an acyclic orientation of
its edges with a unique \emph{source} $s$ and a unique \emph{sink}
$t$, \ie such that $s$ is the only vertex without incoming edge, and
$t$ the only one without outgoing edge; the vertices $s$ and $t$ are
the \emph{poles} of the orientation. Alternative definitions,
characterizations, and several properties are given by De Fraysseix
\emph{et al} in~\cite{Oss}.  Bipolar orientations are a powerful
combinatorial structure and prove insightful to solve many algorithmic
problems such as planar graph embedding~\cite{Le66,Chiba} and
geometric representations of graphs in various flavours (\eg
visibility~\cite{TaTo}, floor planning~\cite{Kant}, straight-line
drawing~\cite{TaTo2,Fu06}).  Thus, it is an interesting issue to have a
better understanding of their combinatorial properties.

\smallskip

This article focuses on the enumeration of bipolar orientations in the
planar case: we consider bipolar orientations on planar maps, where a
\emph{planar map} is a connected graph embedded in the plane (i.e.,
drawn with no edge-intersection, the drawing being considered up to
isotopy).  A \emph{plane bipolar orientation} is a pair $(M,X)$, where
$M$ is a planar map and $X$ is a bipolar orientation of~$M$ having its
poles incident to the outer face of $M$, see Figure~\ref{fig:resume}.
Let $\Theta_{ij}$ be the number of plane bipolar orientations with $i$
non-pole vertices and $j$ inner faces. R.~Baxter proved in~\cite[Eq
  5.3]{baxter} that $\Theta_{ij}$ satisfies the following simple
formula:
\begin{equation}
  \label{eq:theta}
  \Theta_{ij} ~=~ 2 ~
  \frac{(i+j)!~(i+j+1)!~(i+j+2)!}{i!\;(i+1)!\;(i+2)! ~
  j!\;(j+1)!\;(j+2)!}.
\end{equation}
Nevertheless his methodology relies on quite technical algebraic
manipulations on generating functions, with the following steps: the
coefficients $\Theta_{ij}$ are shown to satisfy an explicit recurrence
(expressed with the help of additional ``catalytic" parameters), which
is translated to a functional equation on the associated generating
functions. Then, solving the recurrence requires to solve the
functional equation: Baxter guessed and checked the solution, while
more recently M.~Bousquet-M\'elou described a direct computation
way based on the so-called ``obstinate kernel
method''~\cite{bousquet-melou-four}.

\smallskip

The aim of this article is to give a direct bijective proof of
Formula~(\ref{eq:theta}).  Our main result, 
Theorem~\ref{theo:bijection}, is the description of a bijection
between plane bipolar orientations and certain triples of lattice
paths, illustrated in Figure~\ref{fig:resume}.

\begin{theorem}
  \label{theo:bijection}
  Plane bipolar orientations with $i$ non-pole vertices and $j$ inner
  faces are in bijection with non-intersecting triples of upright
  lattice paths on $\mathbb{Z}^2$ with respective origins $(-1,1)$,
  $(0,0)$, $(1,-1)$, and respective endpoints $(i-1,j+1)$, $(i,j)$,
  $(i+1,j-1)$.
%Formula~(\ref{eq:formul}) follows.
\end{theorem}

\begin{figure}
  \centering
  \includegraphics[width=8cm]{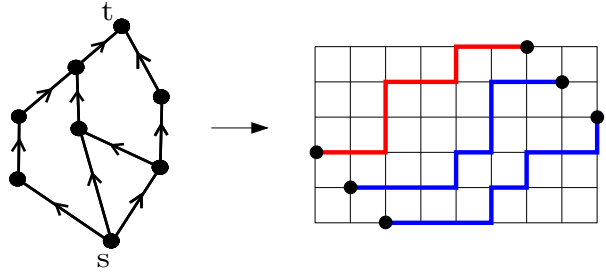}
  \caption{A plane bipolar orientation and the associated triple of
    non-intersecting upright lattice paths.}\label{fig:resume}
\end{figure}

This constitutes a proof of Formula~(\ref{eq:theta}), since the
latter is easily derived from Theorem~\ref{theo:bijection} using the
Gessel--Viennot Lemma~\cite{GeVi1,GeVi2}:

\begin{lemma}[Gessel--Viennot]
  Let $k$ be a positive integer, $\A = \{A_1,\ldots,A_k\}$ and $\B
  = \{B_1,\ldots,B_k\}$ be two sets of points on the $\mathbb{Z}^2$
  lattice, such that any $k$-tuple of non-intersecting upright lattice
  paths with starting points in \A and endpoints in \B necessarily
  join together $A_p$ and $B_p$ for any index~$p$. Then the number of
  such $k$-tuples~is:
  \[
  \Theta=\mathrm{Det}(M),
  \]
  where $M$ is the $k\times k$ matrix such that $M_{pq}$
  is the number of upright lattice paths from $A_p$ to~$B_q$.
\end{lemma}

By Theorem~\ref{theo:bijection}, $\Theta_{ij}$ is equal to 
the number of triples of non-intersecting lattice paths from
$A_1=(-1,1),A_2=(0,0),A_3=(1,-1)$ to
$B_1=(i-1,j+1),B_2=(i,j),B_3=(i+1,j-1)$.  Hence,

\[
\Theta_{ij} ~=~ \begin{array}{|ccc|}
\dbinom{j+i}{i} & \dbinom{j+i}{i+1} & \dbinom{j+i}{i+2} \\&&\\
\dbinom{j+i}{i-1}  & \dbinom{j+i}{i} & \dbinom{j+i}{i+1} \\&&\\
\dbinom{j+i}{i-2}  & \dbinom{j+i}{i-1}  & \dbinom{j+i}{i} 
\end{array}
~=~ \frac{2 \;\; (i+j)! \;\; (i+j+1)! \;\; (i+j+2)!}{i! \; (i+1)! \; (i+2)! \;\;
  j! \; (j+1)!\; (j+2)!}.
\]

\medskip

The second main result of this paper is to show that our bijection
extends in a natural way a bijection that has been recently described
by Bernardi and Bonichon~\cite{BeBo07} (which itself reformulates an
original construction due to Bonichon~\cite{B02}) to count another
well-known and powerful combinatorial structure related to planar
maps, namely Schnyder woods on triangulations
\cite[Chapter~2]{FeBook}. Actually our construction draws much of its
inspiration from the one in~\cite{BeBo07}. We recover the
correspondence between these Schnyder woods and non-crossing pairs of
Dyck paths, which easily yields the formula
\begin{equation}
  S_n ~=~ C_nC_{n+2}\;-\;C_{n+1}^2 ~=~
  \frac{6 \;\; (2n)! \;\; (2n+2)!}{n! \; (n+1)! \; (n+2)! \; (n+3)!}
\end{equation}
for the number $S_n$ of Schnyder woods on triangulations with $n$
inner vertices (where $C_n$ denotes the $n$th Catalan
number $(2n)!/(n!(n+1)!)$).

%%%%%%%%%%%%%%%%%%%%%%%%%%%%%%%%%%%%%%%%%%%%%%%%%%%%%%%%%%%%%%%%%%%%%%%%%%%%
\subsubsection*{Recent related work.} 
%%%%%%%%%%%%%%%%%%%%%%%%%%%%%%%%%%%%%%%%%%%%%%%%%%%%%%%%%%%%%%%%%%%%%%%%%%%%
Felsner et al~\cite{FeFuNoOr07} have very recently exhibited a whole
collection of combinatorial structures that are bijectively related
with one another, among which plane bipolar orientations, separating
decompositions on quadrangulations, Baxter permutations, and triples
of non-intersecting paths.  Though very close in spirit, our bijection
is not equivalent to the one exhibited in~\cite{FeFuNoOr07}.
%, which treats the two trees of a separating decomposition in a
%symmetric way. Another difference is that 
In particular, the restriction of this bijection %in~\cite{FeFuNoOr07}
to count Schnyder woods is a bit more involved than our one and is not
equivalent to the bijection of Bernardi and Bonichon~\cite{BeBo07}.

Even more recently, Bonichon et al~\cite{BoBoFu08} have described a
simple and direct bijection between plane bipolar orientations and
Baxter permutations. These Baxter permutations are known to be encoded
by non-intersecting triples of lattice paths since work by Dulucq and
Guibert~\cite{DuGu98}.  Combining the bijections in~\cite{BoBoFu08}
and~\cite{DuGu98} leads to yet another bijection (almost equivalent to
the one in~\cite{FeFuNoOr07}) between plane bipolar orientation and
non-intersecting triple of paths.

%%%%%%%%%%%%%%%%%%%%%%%%%%%%%%%%%%%%%%%%%%%%%%%%%%%%%%%%%%%%%%%%%%%%%%%%%%%%
\subsubsection*{The main steps to encode a plane bipolar orientation
by a non-intersecting triple of paths.} 
%%%%%%%%%%%%%%%%%%%%%%%%%%%%%%%%%%%%%%%%%%%%%%%%%%%%%%%%%%%%%%%%%%%%%%%%%%%%
At first (Section~\ref{sec:reduce}), we recall a well-known bijective
correspondence between plane bipolar orientations and certain
decompositions of quadrangulations into two spanning trees, which are
called \emph{separating decompositions}.  The next step
(Section~\ref{sec:triple}) is to encode such a separating
decomposition by a triple of words with some prefix conditions: the
first two words encode one of the two trees $T$, in a slight variation
on well known previous results for the 2-parameter enumeration of
plane trees or binary trees (counted by the so-called Narayana numbers).
%, which often leads to non-intersecting pairs of lattice paths and
%the so-called Narayana numbers, see~\cite[A001263]{Sloane}.
The third word encodes the way the edges of the other tree shuffle in
the tree $T$.  The last step (Section~\ref{sec:representation}) of the
bijection is to represent the triple of words as a triple of upright
lattice paths, on which the prefix conditions translate into a
non-intersecting property.

%%%%%%%%%%%%%%%%%%%%%%%%%%%%%%%%%%%%%%%%%%%%%%%%%%%%%%%%%%%%%%%%%%%%%%%%%%%%
\section{Reduction to counting separating decompositions on quadrangulations}
\label{sec:reduce}
%%%%%%%%%%%%%%%%%%%%%%%%%%%%%%%%%%%%%%%%%%%%%%%%%%%%%%%%%%%%%%%%%%%%%%%%%%%%
A \emph{quadrangulation} is a planar map with no loop nor multiple
edge and such that all faces have degree~4. Such maps correspond to
\emph{maximal} bipartite planar maps, \ie bipartite planar maps that
would not stay bipartite or planar if an edge were added between two of
their vertices.

Let $O=(M,X)$ be a plane bipolar orientation; the
\emph{quadrangulation} $Q$ of $M$ is the bipartite map obtained as
follows: say vertices of $M$ are black, and put a white vertex in each
face of $M$; it proves convenient in this particular context to define
a special treatment for the outer face, and put two white vertices in
it, one on the left side and one on the right side of $M$ when the
source and sink are drawn at the bottom and at the top, respectively. 
These black and white
vertices are the vertices of $Q$, and the edges of $Q$ correspond to
the incidences between vertices and faces of $M$. This construction,
which can be traced back to Brown and Tutte~\cite{BT64}, is
illustrated in Figure~\ref{fig:tutte}.  It is well known that $Q$ is
indeed a quadrangulation: to each edge $e$ of $M$ corresponds an inner
(\ie bounded) face of $Q$ (the unique one containing $e$ in its
interior), and our particular treatment of the outer face also
produces a quadrangle.

\begin{figure}
  \centering
  \includegraphics[width=12cm]{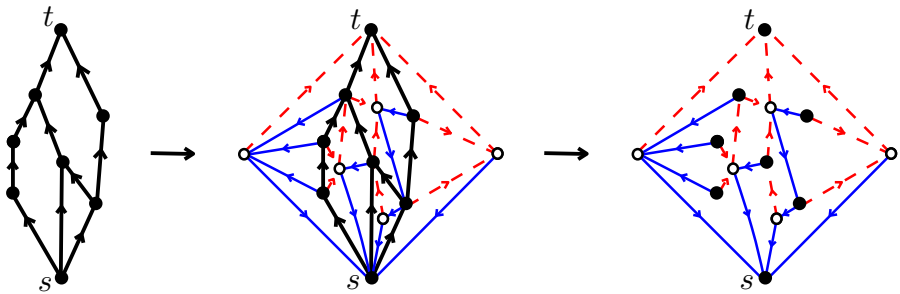}
  \caption{From a plane bipolar orientation to 
    a separating decomposition.}
  \label{fig:tutte}
\end{figure}

If $M$ is endowed with a bipolar orientation $O$, this classical
construction can be enriched to transfer the orientation on~$Q$, as
shown in Figure~\ref{fig:tutte}.  Notice that $O$ (or, in general, any
plane bipolar orientation) satisfies the two following local
conditions~\cite{DeOss} illustrated in Figure~\ref{fig:local_rules_a},
as easily proved using the acyclicity of the orientation and the
Jordan curve theorem:
\begin{itemize}
\item
  edges incident to a non-pole vertex are partitioned into a non-empty
  block of incoming edges and a non-empty block of outgoing edges,
\item
  dually, the contour of each inner face $f$ consists of two oriented
  paths (one path has $f$ on its left, the other one has $f$ on its
  right); the common extremities of the paths are called the two
  \emph{extremal vertices} of $f$. 
 %% the one with $f$ on its right
 %% (left) is called the \emph{left lateral path} (\emph{right lateral
 %% path, \resp}) of $f$ two poles~\cite{DeOss};
\end{itemize}

\begin{figure}
  \psset{unit=2.1em}
  \centering 
  \subfigure[in a plane bipolar orientation,\label{fig:local_rules_a}]{%
    \pspicture(9,4)
    \rput[bl](0,0){\includegraphics[width=9\psunit]{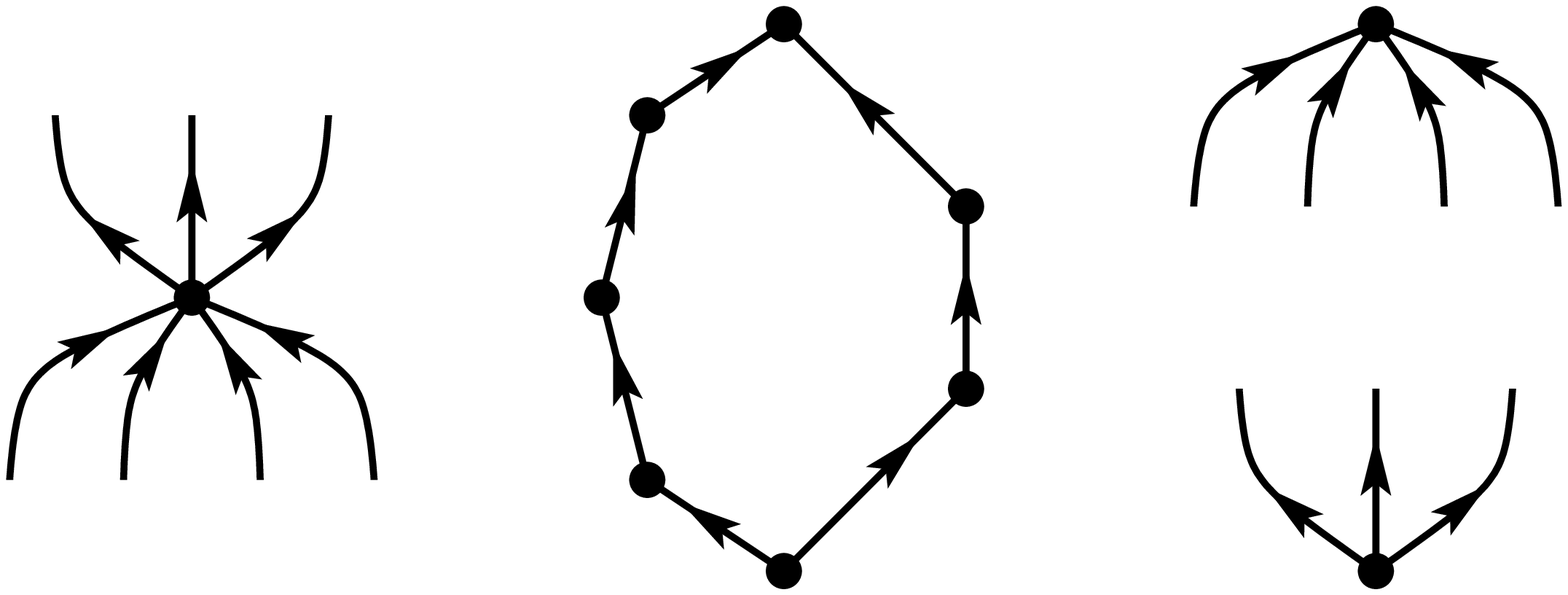}}
    \rput[b](8,0){$s$}\rput[t](8,4){$t$}
    \endpspicture}
  \quad\vline\quad
  \subfigure[and in a separating decomposition.\label{fig:local_rules_b}]{%
    \pspicture(10,4)
    \rput[bl](0,0){\includegraphics[width=10\psunit]{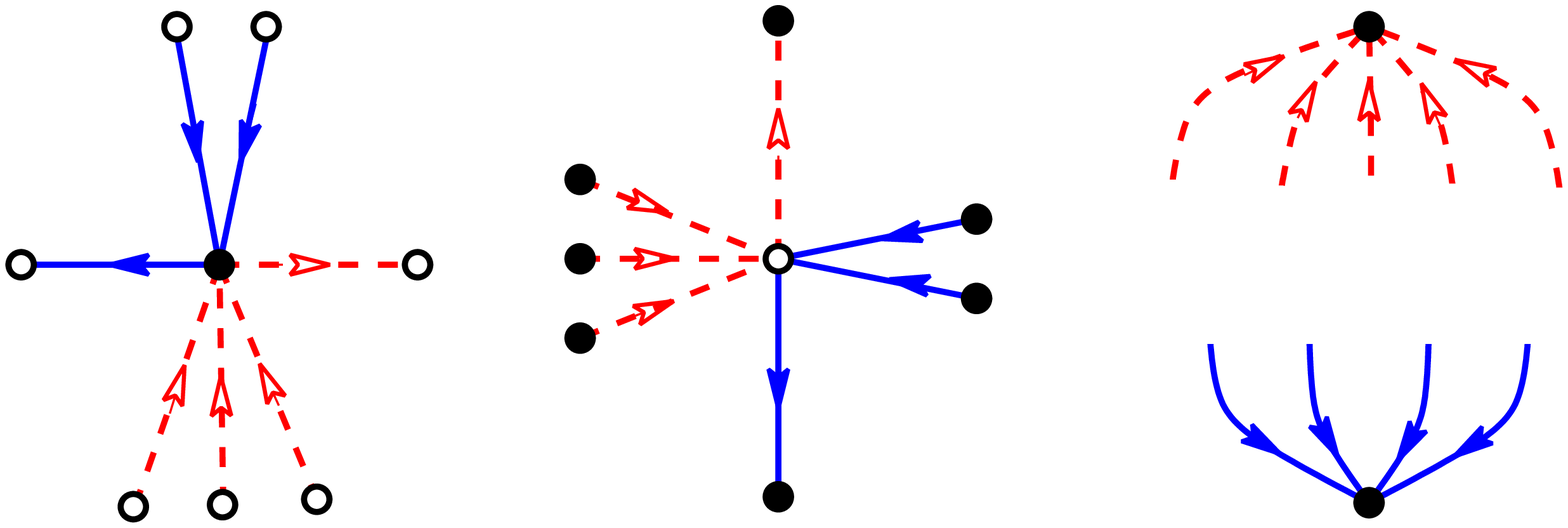}}
    \rput[b](9,0){$s$}\rput[t](9,4){$t$}
    \endpspicture}
  \caption{The local rules.}
\end{figure}

%% Consider a quadrangulation $Q$ bicolored at its vertices (vertices are
%% black or white and each edge connects a black vertex to a white
%% vertex), where the two outer black vertices
%% are denoted $s$ and $t$.

A \emph{separating decomposition} of $Q$ is an orientation and
bicoloration of its edges, say in red or blue, that satisfy the
following local conditions illustrated in
Figure~\ref{fig:local_rules_b} (in all figures, red edges are dashed):
\begin{itemize}
\item each inner vertex has exactly two outgoing edges, a red one and
  a blue one;
\item around each inner black (white, \resp) vertex, the incoming edges
  in each color follow the outgoing one in clockwise (counterclockwise, \resp) order;
\item all edges incident to $s$ are incoming blue, and all edges
  incident to $t$ are incoming red.
\end{itemize}

%\begin{figure}[bt]
%  \centering \includegraphics[width=.5\linewidth]{Figures/local_rules_b.eps}
%  \caption{The local rules in a separating decomposition.}
%  \label{fig:local_rules_b}
%end{figure}

Given an inner face $f$ of $M$, let us orient the two corresponding
edges of $Q$ from the white vertex $w_f$ corresponding to $f$ to the
extremal vertices of $f$, and color respectively in red and blue the
up- and the down-edges. The other edges incident to $w_f$ are oriented
and colored so as to satisfy the circular order condition
around~$w_f$. This defines actually a separating decomposition of $Q$,
and this mapping from plane bipolar orientations to separating
decompositions is one-to-one, as proved by an easy extension
of~\cite[Theorem 5.3]{DeOss}: %%(see also~\cite{FeFuNoOr}):

\begin{proposition}
  \label{theo:bi}
  Plane bipolar orientations with $i$ non-pole vertices and $j$ inner
  faces are in bijection with separating decompositions on
  quadrangulations with $i+2$ black vertices and $j+2$ white vertices.
\end{proposition}

Accordingly, encoding plane bipolar orientations \wrt the numbers of
vertices and faces is equivalent to encoding separating decompositions
\wrt the numbers of black and white vertices.

%%%%%%%%%%%%%%%%%%%%%%%%%%%%%%%%%%%%%%%%%%%%%%%%%%%%%%%%%%%%%%%%%%%%%%%%%%%%
\section{Encoding a separating decomposition by a triple of
  non-intersecting paths}
\label{sec:encode}
%%%%%%%%%%%%%%%%%%%%%%%%%%%%%%%%%%%%%%%%%%%%%%%%%%%%%%%%%%%%%%%%%%%%%%%%%%%%
Separating decompositions have an interesting property: as shown
in~\cite{Bi,Hu}, blue edges form a tree spanning all vertices but $t$,
and red edges form a tree spanning all vertices but $s$. Moreover,
the orientation of the edges corresponds to the natural orientation
toward the root in both trees (the root is $s$ for the blue tree and $t$ 
for the red tree).

%%%%%%%%%%%%%%%%%%%%%%%%%%%%%%%%%%%%%%%%%%%%%%%%%%%%%%%%%%%%%%%%%%%%%%%%%%%%
\subsection{From a separating decomposition to a triple of words}\label{sec:triple}
%%%%%%%%%%%%%%%%%%%%%%%%%%%%%%%%%%%%%%%%%%%%%%%%%%%%%%%%%%%%%%%%%%%%%%%%%%%%
Let $D$ be a separating decomposition with $i+2$ black vertices and
$j+2$ white vertices, and let $\Tb$ be its blue tree.  A
\emph{clockwise} (or shortly \emph{cw}) \emph{traversal} of a tree is
a walk around the tree with the outer face on the left.  We define
the \emph{contour word} $W_Q$ of $Q$ as the word on the alphabet
$\{a,\ua,b,\ub,c,\uc\}$ that encodes the clockwise traversal of $\Tb$
starting at $s$ in the following manner
%and writing letters as follows 
(see Figure~\ref{fig:contour}): letter $a$ ($b$, \resp) codes the
traversal of an edge $e$ of $\Tb$ from a black to a white vertex (from
a white to a black one, \resp), and the letter is underlined if it
corresponds to the second traversal of $e$; letter $c$ codes the
crossing of red edge at a white vertex, and is underlined it if the
edge is incoming.

\begin{figure}
  \small
  \psset{angle=-90, nodesep=1pt}
  \def\a(#1,#2){\rput[b](#1,#2){$a$}}
  \def\A(#1,#2){\rput[b](#1,#2){$\ua$}}
  \def\b(#1,#2){\rput[b](#1,#2){$b$}}
  \def\B(#1,#2){\rput[b](#1,#2){$\ub$}}
  \def\c(#1,#2){\rput[b](#1,#2){$c$}}
  \def\C(#1,#2){\rput[b](#1,#2){$\uc$}}
  \def\Wa{\a(11,0.5)\A(5.5,12)\A(5,8.5)\A(6.5,5)\a(12,1.5)\A(14,9.5)\a(13.5,2)\a(19,11)\A(21.5,12.5)\A(20,8)\a(15.5,2)}
  \def\Wb{\b(1.5,9)\b(2.5,8.5)\b(3,7)\B(4,4.75)\b(11.5,10)\B(11.5,7)\b(17.5,7)\b(19,15)\B(19.5,13)\B(17,4)\B(25.5,7)}
  \def\Wc{\c(0.5,8)\C(10,7)\C(9.5,8)\c(10,9.5)\c(16,6.5)\C(17.25,13)\C(16,14)\c(17,15.5)\C(24,8)\C(24.5,9.5)\c(26.5,10)}
  \def\WQ{\rput[t](13.5,-1){$W_Q = ac b\ua b\ua b\ua \ub a \uc\uc c b\ua \ub acba\uc\uc cb\ua\ub\ua\ub a\uc\uc c\ub$}}
  \def\Wt{\rput[t](13.5,-1){$W_t = a b\ua b\ua b\ua \ub ab\ua \ub abab\ua\ub\ua\ub a\ub$}}
  \def\Wm{\rput[t](13.5,-1){$W_m = ac \rnode{p1}{\ua} \rnode{p2}{\ua} \rnode{p3}{\ua} a%
    \rnode{q3}{\uc}\rnode{q2}{\uc} c \rnode{p4}{\ua}%
    aca\rnode{q4}{\uc}\rnode{q1}{\uc} c\rnode{p5}{\ua}\rnode{p6}{\ua}%
    a\rnode{q6}{\uc}\rnode{q5}{\uc} c$}%
    \ncbar[arm=6pt]{p1}{q1}\ncbar[arm=4pt]{p2}{q2}\ncbar[arm=2pt]{p3}{q3}%
    \ncbar[arm=2pt]{p4}{q4}\ncbar[arm=4pt]{p5}{q5}\ncbar[arm=2pt]{p6}{q6}%
}
  \def\complet{\pspicture(0,-2)(27,23)%
    \rput[bl](0,0){\includegraphics[width=27\psunit]{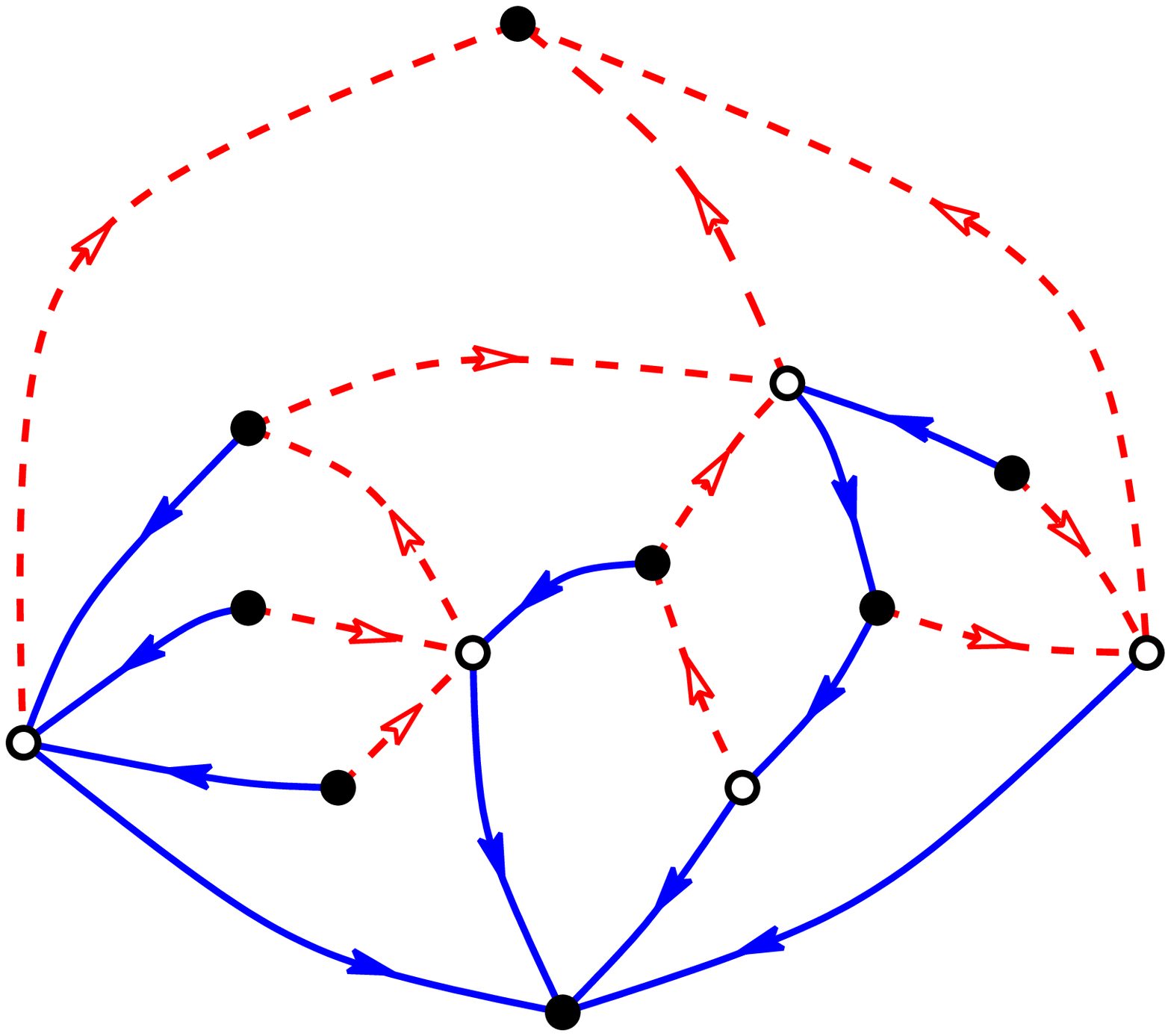}}%
    \Wa\Wb\Wc\WQ%
    \endpspicture}
  \def\bleu{\pspicture(0,-2)(27,16)%
    \rput[bl](0,0){\includegraphics[width=27\psunit]{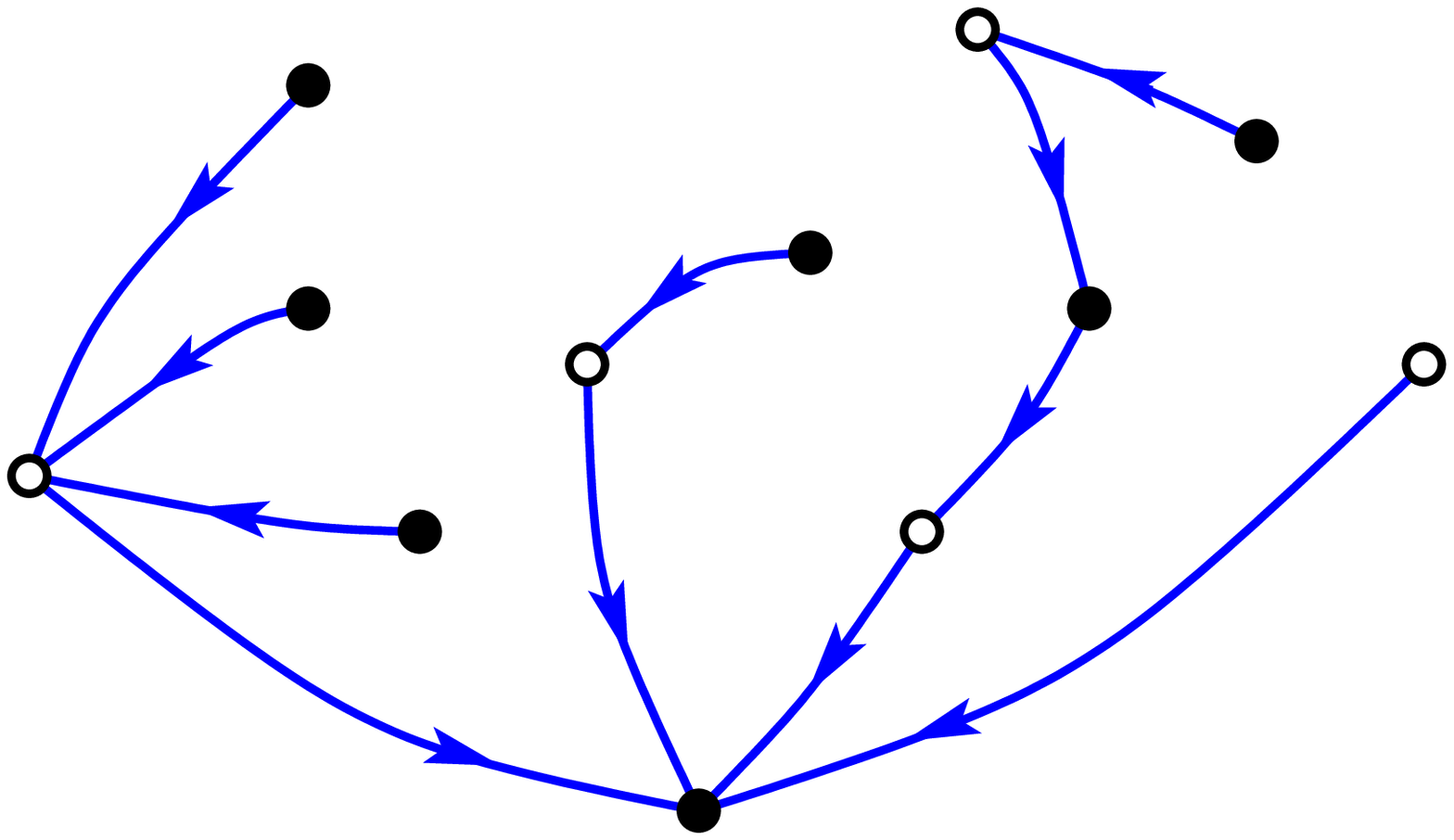}}%
    \Wa\Wb\Wt\endpspicture}
  \def\rouge{\pspicture(0,-2)(27,15)%
    \rput[bl](0,0){\includegraphics[width=27\psunit]{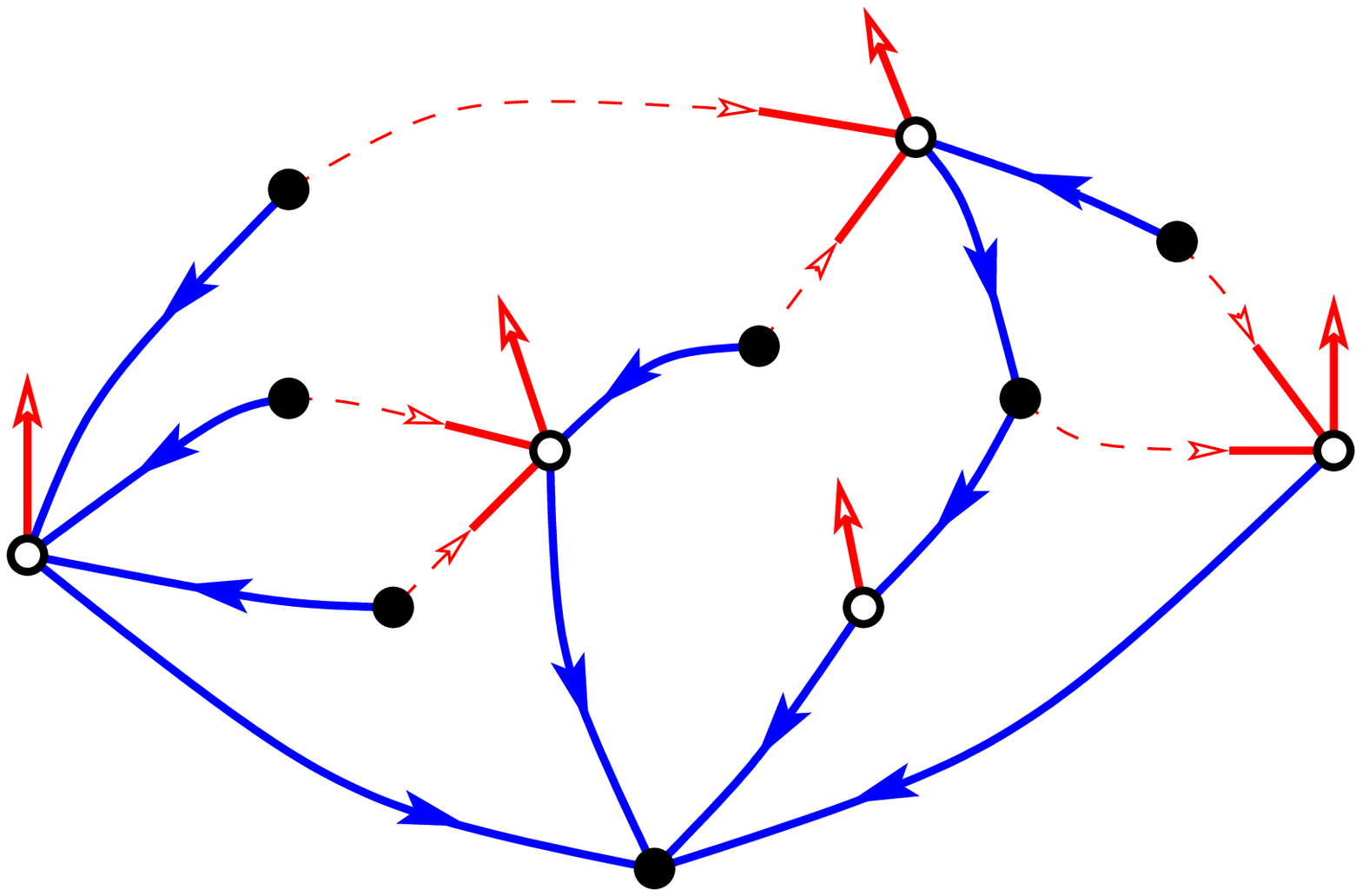}}%
    \Wa\Wc\Wm\endpspicture}
  %%%%
  \centering
  \begin{pspicture}(64,43)%\mygrid
    \rput[bl](0,11){\complet}
    \rput[tr](64,43){\bleu}
    \rput[br](64,2){\rouge}
    \psline{->}(29,20)(35,13)
    \psline{->}(29,24)(35,32)
  \end{pspicture}
  \caption{The words encoding a separating decomposition.}
  \label{fig:contour}
\end{figure}

We shall consider three subwords of $W_Q$: for any $\ell$ in
$\{a,b,c\}$, let $W_{\ell}$ denote the subword obtained by keeping
only the letters in the alphabet $\{\ell,\ul\}$. In order to describe
the properties of these words, we also introduce the \emph{tree-word}
$W_{\mathrm{t}}$ and the \emph{matching word} $W_{\mathrm{m}}$, that
are respectively obtained from $W_Q$ by keeping the letters in
$\{a,\ua,b,\ub\}$, and in $\{a,\ua,c,\uc\}$.

%%%%%%%%%%%%%%%%%%%%%%%%%%%%%%%%%%%%%%%%%%%%%%%%%%%%%%%%%%%%%%%%%%%%%%%%%%%%
\subsubsection{The tree-word encodes the blue tree.}\label{sec:treeWord}
%%%%%%%%%%%%%%%%%%%%%%%%%%%%%%%%%%%%%%%%%%%%%%%%%%%%%%%%%%%%%%%%%%%%%%%%%%%%
Observe that $W_{\mathrm{t}}$ corresponds to a classical Dyck encoding
of $\Tb$, in which the two alphabets $\{a,\ua\}$ and $\{b,\ub\}$ are
used alternatively to encode the bicoloration of vertices. Hence
$W_{\mathrm{t}}$ is just obtained by interlacing $W_a$ and $W_b$
starting with $a$, and each prefix of $W_{\mathrm{t}}$ has at least as
many non-underlined letters as underlined letters.

Let us count precisely the number of occurrences of letters $a$,
$\ua$, $b$ and $\ub$ in $W_{\mathrm{t}}$. For this purpose, let us
associate each edge of a tree with its extremity that is farther from
the root.
%(this yields a matching between the edges and the non-root vertices of the tree).
From the defining rules it follows that the two traversals of edges
corresponding to black vertices are encoded by $b$ and $\ua$, while
those of edges corresponding to white vertices are encoded by $a$ and
$\ub$.  In other words, each occurrence of a letter $a$, $\ua$, $b$,
$\ub$ corresponds to the first visit to a white vertex, last visit to
a black vertex, first visit to a black vertex, and last visit to a
white vertex, respectively. As $\Tb$ has $i$ non-root black vertices
and $j+2$ white vertices, the word $W_a$ has $j+2$ occurrences of $a$
and $i$ occurrences of $\ua$, shortly written
$W_a\in\mathfrak{S}(a^{j+2}\ua^i)$. Similarly,
$W_b\in\mathfrak{S}(b^i\ub^{j+2})$.  Furthermore, the fact that each
prefix of $W_{\mathrm{t}}$ has at least as many non-underlined letters
as underlined letters translates into the following property for the
pair $(W_a,W_b)$:
 
\begin{property}\label{property1}
  For $1\leq k\leq i$, the number of $a$'s on the left of the $k$th
  occurrence of $\ua$ in $W_a$ is strictly larger than the number of
  $\ub$'s on the left of the $k$th occurrence of $b$ in $W_b$.
\end{property}    

\begin{proof}
\cacher{ Assume that $W_{\mathrm{t}}$ is not a Dyck word, and consider
  the shortest prefix of $W_{\mathrm{t}}$ having more underlined
  letters than non-underlined letters. By minimality, the last letter
  of the prefix has to be underlined and is at an odd position
  $2\ell+1$, so that this letter is an $\ua$. By minimality also, the
  prefix $w_{2\ell}$ of length $2\ell$ has the same number of
  non-underlined letters as underlined letters. Moreover, $w_{2\ell}$
  has $\ell$ letters in $\{a,\ua\}$ and $\ell$ letters in $\{b,\ub\}$,
  because the letters of type $\{a,\ua\}$ alternate with letters of
  type $\{b,\ub\}$. Hence, if we denote by $k$ the number of $\ua$'s
  in $w_{2\ell}$, then $w_{2\ell}$ has $\ell-k$ occurrences of $a$,
  $k$ occurrences of $b$, and $\ell-k$ occurrences of $\ub$. In
  particular, the number of occurrences of $a$ on the left of the
  $(k+1)$th occurrence of $\ua$ in $W_a$ is $(\ell-k)$, and the number
  of occurrences of $\ub$ on the left of the $(k+1)$th occurrence of
  $b$ is at least $(\ell-k)$. This contradicts
  Property~\ref{property1} of an admissible triple of words.% 
  }  
%%%%%
  For each $k$, let $N_a(k)$ and $N_{\ub}(k)$ be the numbers of $a$'s
  and $\ub$'s in $W_{\mathrm{t}}$ on the left of the $k$th
  occurrence of $\ua$ (\resp  $b$).  Let $p$ be the
  prefix of $W_{\mathrm{t}}$ ending just before the $k$th occurrence
  of $\ua$. Notice that $p$ ends at a letter in $\{b,\ub\}$, so $p$
  has even length $2m$ with $m$ letters in $\{a,\ua\}$ and $m$ letters
  in $\{b,\ub\}$. Let $m_a$, $m_{\ua}$, $m_b$, $m_{\ub}$ be
  respectively the numbers of $a$'s, $\ua$'s, $b$'s, and $\ub$'s in
  $p$ (notice that $m_{\ua}=k-1$ and $m_a=N_a(k)$).  Since $\Wt$ is a
  Dyck word and since $p$ is followed by an underlined letter, we have
  $m_a+m_b> m_{\ua}+m_{\ub}$.  But $m_{\ua}=m-m_a$ and
  $m_b=m-m_{\ub}$, so we obtain both (i): $m_{\ub} < m_a=N_a(k)$ and
  (ii): $m_b> m_{\ua}=k-1$. From (ii) the $k$th occurrence of $b$ in
  $\Wt$ belongs to $p$, and from (i) the number $N_{\ub}(k)$ of
  $\ub$'s on its left is strictly smaller than $N_a(k)$. This
  concludes the proof.
\end{proof}

The words $W_a$ and $W_b$ have the additional property that two
letters are redundant in each word.  Indeed, the first and the last
letter of $W_a$ are $a$'s and the last two letters of $W_b$ are
$\ub$'s, because of the rightmost branch of $\Tb$ being reduced to an
edge, see Figure~\ref{fig:contour}.

%%%%%%%%%%%%%%%%%%%%%%%%%%%%%%%%%%%%%%%%%%%%%%%%%%%%%%%%%%%%%%%%%%%%%%%%%%%%

%\noindent{\bf Remark.} Another property satisfied by the word $W_a$ is 
%that it ends with a letter $a$.
%This corresponds to the property that the rightmost branch of $\Tb$ is
%reduced to an edge, see Figure~\ref{fig:contour}(b).

%%%%%%%%%%%%%%%%%%%%%%%%%%%%%%%%%%%%%%%%%%%%%%%%%%%%%%%%%%%%%%%%%%%%%%%%%%%%
\subsubsection{The matching word encodes the red edges.}
Let us now focus on $W_c$ and on the matching word
$W_{\mathrm{m}}$. Clearly, any occurrence of a letter $c$ ($\uc$) in
$W_Q$ corresponds to a red edge with white (black, \resp) origin, see
Figure~\ref{fig:contour}. Hence
$W_c\in\mathfrak{S}(c^{j+2}\uc^i)$. Moreover $W_c$ starts and ends
with a letter $c$, corresponding to the two outer red edges.

Observe also that any occurrence of $a$ in $W_{\mathrm{m}}$, which
corresponds to the first visit to a white vertex $v$, is immediately
followed by a pattern $\uc^{\ell}c$, with $\ell$ the number of
incoming red edges at $v$. Hence $W_{\mathrm{m}}$ satisfies the
regular expression:
\begin{equation}\label{eq:Wm}
W_{\mathrm{m}}\in ac(\ua^*a\uc^*c)^*,
\end{equation}
where $E^*$ denotes the set of all (possibly empty) sequences of elements from
$E$. Notice that this property uniquely defines $W_{\mathrm{m}}$ as a
shuffle of $W_a$ and~$W_c$.

\begin{lemma}\label{lem:orderinred}
Let $S$ be a separating decomposition, with $\Tb$ the tree induced by
the blue edges.  Consider a red edge $e$ of $S$ not incident to
$t$, with $b$ ($w$) the black (white, \resp) extremity of $e$. 
Then the last visit to $b$
occurs before the first visit to $w$ during a cw traversal around
$\Tb$ starting at $s$.
\end{lemma}

\begin{proof}
First, the local conditions of separating decompositions ensure that
$e$ is connected to $b$ ($w$) in the corner corresponding to the last
visit to $b$ (first visit to $w$, \resp).  Hence we just have to prove
that, if $C$ denotes the unique simple cycle formed by $e$ and edges
of the blue tree, then the edge $e$ is traversed from $b$ to $w$ when
walking cw around $C$.  Assume \emph{a contrario} that $e$ is
traversed from $w$ to $b$ during a cw walk around $C$.  If $e$ is
directed from $b$ to $w$ (the case of $e$ directed from $w$ to $b$ can
be treated similarly), then the local conditions of separating
decompositions ensure that the red outgoing path $P(w)$ of $w$ (\ie
the unique oriented red path that goes from $w$ to $t$) starts going
into the interior of $C$. According to the local conditions, no
oriented red path can cross the blue tree, hence $P(w)$ has to go out
of $C$ at $b$ or at $w$: going out at $w$ is impossible as it would
induce a red circuit, going out at $b$ contradicts the local
conditions; hence either case yields a contradiction.
\end{proof}

Let us now consider a red edge $e=(b,w)$ with a black origin.  The
outgoing half-edge of $e$ is in the corner of the last visit to $b$,
encoded by a letter $\ua$, while the incoming half-edge of $e$, which
is encoded by a letter $\uc$, is in the corner of the first visit to
$w$.  Hence, according to Lemma~\ref{lem:orderinred}, the $\ua$ occurs
before the $\uc$.  In other words, the restriction of $W_{\mathrm{m}}$
to the alphabet $\{\ua, \uc\}$ is a parenthesis word (interpreting
each $\ua$ as an opening parenthesis and each $\uc$ as a closing
parenthesis), and each parenthesis matching corresponds to a red edge
with a black origin, see Figure~\ref{fig:contour}. According to the 
correspondence between the $a$'s and the $c$'s 
(see the regular expression~\eqref{eq:Wm} of $W_{\mathrm{m}}$), this parenthesis
property of $W_{\mathrm{m}}$ is translated as follows:

\begin{property}\label{property2}
  For $1\leq k\leq j+2$, the number of $\ua$'s on the left of the
  $k$th occurrence of $a$ in $W_a$ is at least as large as the number
  of $\uc$'s on the left of the $k$th occurrence of $c$ in $W_c$.
\end{property}

\noindent{\bf Definition.}  A triple of words $(W_a, W_b, W_c)$ in
$\mathfrak{S}(a^{j+2}\ua^i) \times \mathfrak{S}(b^i\ub^{j+2}) \times
\mathfrak{S}(c^{j+2}\uc^i)$ is said to be \emph{admissible of type
  $(i,j)$} if $W_a$ ($W_c$, \resp) ends with a letter $a$ ($c$, \resp )
and if Property~\ref{property1} and Property~\ref{property2} are
satisfied.

\medskip

Observe that this definition yields other redundant letters, namely,
$W_a$ has to start with a letter $a$, $W_c$ has to start with a letter
$c$, and $W_b$ has to end with two letters $\ub$.

\subsection{From an admissible triple of words to a triple of non-intersecting paths}
\label{sec:representation}
The properties of an admissible triple of words are formulated in a
more convenient way on lattice paths.  This section describes the
correspondence, illustrated in Figure~\ref{fig:paths}.

Consider an admissible triple of words $(W_a,W_b,W_c)$
of type $(i,j)$, 
%Let us delete the last letter in each word,
%which are redundant ($a$ in $W_a$, $\ub$ in $W_b$, $c$ in $W_c$).
% initial and
%terminal $a$ in $W_a$, last two $\ub$ in $W_b$, and initial and
%terminal $c$ in $W_c$.
and represent each word as an upright lattice path starting at the
origin, the binary word being read from left to right, and the
associated path going up or right depending on the letter. The letters
associated to up steps are $a$, $\ub$ and $c$. Clearly, as $(W_a, W_b,
W_c) \in \mathfrak{S}(a^{j+2}\ua^i) \times \mathfrak{S}(b^i\ub^{j+2}) \times
\mathfrak{S}(c^{j+2}\uc^i)$, the three paths end at $(i,j+2)$.

\begin{figure}
  \centering
  \includegraphics[width=14cm]{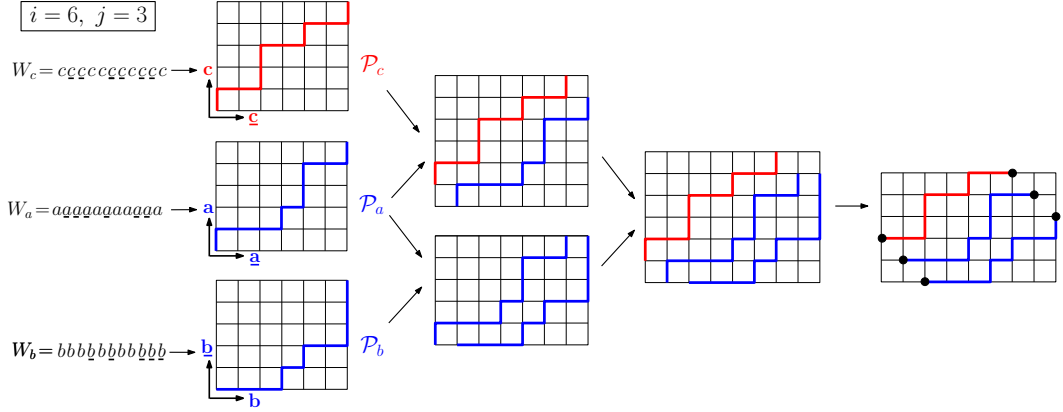}
  \caption{An admissible triple of words is naturally represented as a non-intersecting triple of paths.}
  \label{fig:paths}
\end{figure}

Property~\ref{property1} is translated into: 
\begin{quote}
``for $1\leq k\leq i$, the
$k$th horizontal step of $\mathcal{P}_a$ (ending at abscissa $k$) is
strictly above the $k$th horizontal step of $\mathcal{P}_b$.''
\end{quote}
Hence,
Property~\ref{property1} is equivalent to the fact that $\mathcal{P}_a$ and the 
shift of $\mathcal{P}_b$ one step to the right are non-intersecting.  

Similarly,
Property~\ref{property2} is translated into: 
\begin{quote}
``for $1\leq k\leq j+2$, the $k$th
vertical step of $\mathcal{P}_a$ is weakly on the right of the $k$th
vertical step of $\mathcal{P}_c$.''
\end{quote}
In other words, $\mathcal{P}_c$ is weakly
top left of $\mathcal{P}_a$. Hence, Property~\ref{property2}  is equivalent to the fact that
$\mathcal{P}_a$ and the shift of $\mathcal{P}_c$ one step up-left  are
non-intersecting. Let us now consider the redundant letters; they
correspond to two vertical steps in each path, and removing them leads
to a triple $(\cP_b',\cP_a',\cP_c')$ of non-intersecting upright lattice paths with origins $(-1,1)$, $(0,0)$, 
$(1,-1)$ and endpoints $(i-1,j+1)$, $(i,j)$, $(i+1,j-1)$. Such a triple of paths is called
a \emph{non-intersecting triple of paths of type $(i,j)$}.

%The conditions on the words are thus equivalent to
%the property that the three paths put together form a non-intersecting
%triple of paths. Finally, as the paths have to be non-intersecting,
%the first step of $\mathcal{P}_a$ and $\mathcal{P}_c$ and the last
%step of $\mathcal{P}_b$ are redundant (all up), so that origins and
%endpoints can be modified accordingly (see Figure~\ref{fig:paths}
%right).

%The reformulation in terms of paths is summarized by the
%following lemma.

%\begin{lemma}
%\label{lemma:bij_paths}
%Admissible triples of words of type $(i,j)$ are in bijection with 
%non-intersecting triples $(\mathcal{P}_a, \mathcal{P}_b, \mathcal{P}_c)$ 
%of paths such 
%that $\mathcal{P}_a$ goes from $(1,0)$ to $(i+1,j-1)$, $\mathcal{P}_b$ goes
%from $(2,0)$ to $(i+2,j-1)$, and $\mathcal{P}_c$ goes from $(0,1)$ to $(i,j)$.
%\end{lemma}%

%Finally, Theorem~\ref{theo:bijection} follows from the
%concatenation of Theorem~\ref{theo:bi},
%Proposition~\ref{prop:bij}, and Lemma~\ref{lemma:bij_paths}.

%\vspace{-0.2cm}

To sum up, we have described a mapping $\Phi$ from
separating decompositions with $(i+2)$ black and $(j+2)$
white vertices to non-intersecting triples of paths of type $(i,j)$.
%This
%mapping is proved to be bijective by defining an inverse mapping in a
%way that naturally reverses the operations performed by $\Phi$.

%\subsection{Inverse mapping (sketch).} 

%Next, we
%match outgoing red half-edges at black vertices and incoming red
%half-edges at white vertices following Property~2.
% ensures that the subword of
%$\Wm$ by keeping only the letters in $\{\ua,\uc\}$ is a parenthesis
%word, if $\ua$ (\resp $\uc$) is identified with an opening (\resp
%closing) parenthesis. This results in a matching of red half-edges;
%the red outgoing half-edge at the $k$th black vertex (black vertices
%are ordered \wrt last visit in $\Tb$) is merged with the
%incoming red half-edge associated with the letter $\uc$ matched with
%the $k$th occurence of $\ua$ in $\Wm$.  
%Finally it is easily shown that the red half-edges going out of white
%vertices can be completed in a unique way to edges so as to form only
%quadrangular faces, as illustrated in the figure below.  
%\begin{center}
%  \includegraphics[width=12cm]{Figures/face.eps}
%\end{center}
%By construction, the obtained figure is a quadrangulation endowed with
%a bicolorientation, and the construction is inverse to $\Phi$.

\section{The inverse mapping}
\label{sec:inverse}

As we show in this section, the mapping $\Phi$ is easily checked to be a bijection, as all steps
(taken in reverse order) are invertible.
Start from a non-intersecting triple of paths $(\cP_b',\cP_a',\cP_c')$  of type $(i,j)$,
where $\cP_b'$ goes from $(-1,1)$ to $(i-1,j+1)$, $\cP_a'$ goes from $(0,0)$ to $(i,j)$, and 
$\cP_c'$ goes from $(1,-1)$ to $(i+1,j-1)$.
%that go respectively from $(-1,1)$, $(0,0)$, $(1,-1)$ to $(i-1,j+1)$, $(i,j)$, $(i+1,j-1)$.
Append two up-steps in each of the 3 paths: $\cP_b=\cP_b'\uparrow\uparrow$, $\cP_a=\uparrow\cP_a'\uparrow$, 
$\cP_c=\uparrow\cP_c'\uparrow$.

\subsection{Associate an admissible triple of words to the triple of paths.}
Each of the three paths $(\cP_b,\cP_a,\cP_c)$  is equivalent to a binary
word  on the alphabet $\{u,r\}$, corresponding to the sequence of up and right steps when
traversing the path.  
Let $(W_a,W_b,W_c)$ be the three binary words associated respectively to $(\cP_a,\cP_b,\cP_c)$.
In order to have different alphabets for the three words, we substitute the alphabet $(u,r)$
by $(a,\ua)$ for the word $W_a$, by $(\ub,b)$ for the word $W_b$, and by $(c,\uc)$ for the word $W_c$.
As the triple $(\cP_b,\cP_a,\cP_c)$ is non-intersecting, the triple of words $W_a\in\mathfrak{S}(a^{j+2}\ua^i), W_b\in\mathfrak{S}(b^i\ub^{j+2}), W_c\in\mathfrak{S}(c^{j+2}\uc^i)$ is readily checked to be an admissible triple of words of type $(i,j)$. 

\subsection{Construct the blue tree.} 
Define the \emph{tree-word} $\Wt$ as the word obtained by 
interlacing $W_a$ and
$W_b$ starting with $a$. 

\begin{claim}
The word $\Wt$ is a Dyck word (when seeing each letter in $\{a,b\}$ as opening parenthesis and each letter in $\{\ua,\ub\}$ as closing parenthesis).
\end{claim}
\begin{proof}
Clearly $\Wt$ has the same number of underlined
as non-underlined letters.
Assume that $W_{\mathrm{t}}$ is not a Dyck word, and consider the shortest 
 prefix of $W_{\mathrm{t}}$ having more underlined
letters than non-underlined letters. By minimality, the last letter
of the prefix has to be underlined and is at an odd position $2m+1$,
so that this letter is an $\ua$. By minimality also,
 the prefix $w_{2m}$ of length $2m$ has the same number of non-underlined letters as
underlined letters. Moreover, $w_{2m}$ has $m$ 
letters in $\{a,\ua\}$ and $m$ letters in $\{b,\ub\}$, because the letters
of type $\{a,\ua\}$ alternate with letters of type $\{b,\ub\}$. Hence,
if we denote by $k$ the number of $\ua$'s in $w_{2m}$, then $w_{2m}$
has $m-k$ occurrences of $a$, $k$ occurrences of $b$, and 
$m-k$ occurrences of $\ub$. In particular, the number of occurrences of $a$ on 
the left of the $(k+1)$th occurrence of $\ua$ in $W_a$ is $(m-k)$, and 
the number of occurrences of $\ub$ on the left of the $(k+1)$th occurrence
of $b$ is at least $(m-k)$. This contradicts Property~\ref{property1}.  
\end{proof}

Denote by $\Tb$ the plane tree whose Dyck word  is $\Wt$. Actually, as
we have seen in Section~\ref{sec:treeWord}, 
$\Wt$ is a refined Dyck encoding of $\Tb$ that also
takes account of the number of vertices at even depth, colored black,
and the number of vertices at odd length, colored white.
Precisely, $\Tb$ has $i+1$ black
vertices  and $j+2$ white vertices. Denote by $s$ the (black) root of $\Tb$,
and orient all the edges of $\Tb$ toward the root.

\subsection{Insert the red half-edges.} 
The next step is to insert the red edges. Precisely we first insert 
the red \emph{half-edges} (to be merged into complete red edges).  Define the
 \emph{matching word} $\Wm$ as the unique shuffle of $W_a$
and $W_c$ that satisfies the regular expression $ac(\ua^*a\uc^*c)^*$.
For $1\leq k\leq j+2$, consider the $k$th white vertex $w$ in $\Tb$,  the 
vertices being ordered \wrt the first visit during a cw traversal of $\Tb$ starting at $s$.
Let $\ell\geq 0$ be the
number of consecutive $\uc$'s that follow the $k$th occurrence of $a$
in $\Wm$.  Insert $\ell$ incoming and one outgoing red half-edges (in
clockwise order) in the corner of $\Tb$ traversed during the first
visit to $w$. Then, add an outgoing red half-edge to each black vertex
$b$ in the corner traversed during the last visit to $b$. 
The red half-edges are called \emph{stems} as long as they are not completed
into complete red edges, which is the next step. 
Observe that the local conditions of a separating decomposition are already satisfied  
around each vertex (the pole $t$ is not added yet).

\subsection{Merge the red stems into red edges.}\label{sec:closure}
Next, we
match the outgoing red stems at black vertices and the incoming red
stems (which are always at white vertices). Property~\ref{property2}
 ensures that the restriction of 
$\Wm$ to the alphabet $\{\ua,\uc\}$ is a parenthesis
word, viewing each $\ua$ as an opening parenthesis and each $\uc$ as a closing parenthesis. 
By construction, this word corresponds to walking around $\Tb$ and writing a $\ua$ for each last visit to 
a black vertex and a $\uc$ for each incoming red stem.

This yields a matching of the red half-edges;
the red outgoing half-edge inserted in the corner corresponding to the $k$th black vertex (black vertices
are ordered \wrt the last visit in $\Tb$) is merged with the
incoming red half-edge associated with the letter $\uc$ matched with
the $k$th occurrence of $\ua$ in $\Wm$, see Figure~\ref{fig:face_a}.  
Such an operation is called a \emph{closure}, as it
``closes'' a bounded face $f$ on the right
of  the new red edge $e$. The origin of $e$ is called the \emph{left-vertex} of $f$.

%What we do now is to match outgoing red stems with incoming red stems so
%as to obtain red edges going from a black to a white vertex: these are called
%red edges of \emph{black type}. The fact that Condition P2 is satisfied implies
%that each prefix of $W_{\mathrm{m}}$ has at least as many $\ua$'s as $\uc$'s.
%Hence, if we associate an opening parenthesis to each occurence of $\ua$ 
%and a closing parenthesis to each occurence of $\uc$ in $W_{\mathrm{m}}$, we
%obtain a parenthesis matching between the $\ua$'s and the $\uc$'s. 
%For each matching between an $\ua$ and a $\uc$, we merge the
%outgoing red stem of $\ua$ with the incoming red stem of $\uc$ into an edge.

%The left-to-right order of the $\ua$'s and $\uc$'s
%in $W_{\mathrm{m}}$ corresponds to the order in which the corresponding  
%red stems are
%crossed during a clockwise traversal  of the contour of $\Tb$ 
%starting at the root.
%Hence, as the matching between the $\ua$'s and the $\uc$'s in $W_{\mathrm{m}}$ 
%is a parenthesis matching, the corresponding matchings of pairs of half-edges
%maintains planarity, \ie the
%fusion of two half-edges does not create a crossing.

We perform the closures one by one, following an order consistent with 
the $\ua$'s being matched inductively with the $\uc$'s in $W_{\mathrm{m}}$. In Figure~\ref{fig:contour}, 
this means that the red edges 
with a black origin are processed ``from bottom to top''.  
Observe that the planarity is preserved throughout the closures: the red edges that are completed
are nested in the same way as the corresponding arches in the parenthesis word.
%Finally, Observe that the total number of closures
%to perform is $i$, as each closure is associated with one of the $i$ non-root black vertices in $\Tb$.

%For $k\in[1..i]$, let $F_k$ be the figure obtained after performing the first $k$ closures
%(there are $i$ closures, as each closure is associated with one of the $i$ non-root black vertices in 
%$\Tb$),
%Clearly $F_k$ has $k$ bounded faces as each closure forms a bounded face. that are 
%and let $F=F_i$. In $F$ the bounded faces  are in one-to-one correspondence with the red edges of black type. 
%Moreover, in each face of $F$ (including the outer one), 
%the stems going out of white vertices have no opposite half-edge yet. 

\subsection{Insert the remaining half-edges.}
The last step is to complete the stems going out of white vertices into
complete red edges going into black vertices, so as to obtain a quadrangulation endowed 
with a separating decomposition. 
%We show that
%this is done by completing each half-edge $h$ to an edge connected
%to the left-vertex of the bounded face  
%containing $h$.

%When the red edges of black type are successively processed (merged) in 
%a bottom-to-top order, the following invariant $I$ is satisfied by the
%current figure $F_i$ at each step.

\begin{lemma}
For each $k\in[0..i]$ consider the planar map $F_k$ formed by the blue edges
and the completed red edges 
after $k$ closures have been performed. 
The following invariant holds.
\begin{quote}
(I): ``Consider any pair  $c_w,c_b$ of consecutive corners of $F_k$ during
 a ccw traversal of the outer face of $F_k$ (i.e., with the outer face on the right), 
 such that $c_w$ is incident to
a white vertex (thus $c_b$ is incident to a black vertex). Then exactly one of
the two corners contains an outgoing (unmatched) stem.''
\end{quote}
\end{lemma}
\begin{proof}
Induction on $k$.
At the initial step, $F_0$ is the tree $\Tb$. The red stems are inserted in the corners
of $\Tb$---as described in Section~\ref{sec:closure}---in a way that satisfies the local conditions
of separating decompositions. 
Hence it is an easy exercise to check that 
 $F_0$ satisfies $(I)$.
Now assume that, for $k\in[0..i-1]$, $F_k$ satisfies $(I)$, and let us show that the 
same holds for $F_{k+1}$.
Consider the closure that is performed from $F_k$ to $F_{k+1}$.
This closure completes a red edge $e=(b,w)$, where $e$ starts from the corner $c_b$ at
the last visit to $b$ and ends at the corner $c_w$ at the first visit to $w$.
As we see in Figure~\ref{fig:face_a}, the closure expels all the 
corners strictly between $c_w$ and $c_b$ from the outer face, and it makes $c_b$ the new follower of $c_w$. 
According to the local conditions of separating decompositions, $c_w$
contains an outgoing stem in the outer face of $F_{k+1}$. In addition, $c_b$ contains no
outgoing stem in $F_{k+1}$, because the outgoing stem of $b$ is matched by the closure.
Hence, $F_{k+1}$ satisfies $(I)$. 
\end{proof}

\begin{figure}
\centering
\subfigure[\label{fig:face_a}]{\includegraphics[height=4em]{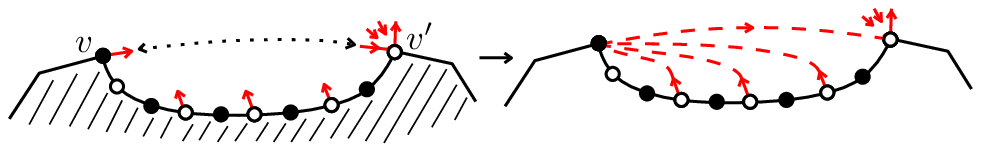}}
\qquad\qquad
\subfigure[\label{fig:face_b}]{\includegraphics[height=7em]{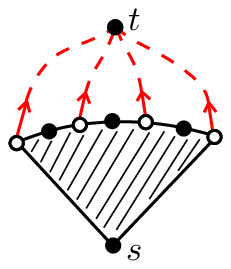}}
\caption{Completing the red stems going out of white vertices.}
\label{fig:face}
\end{figure}

Denote by $F=F_i$ the figure that is 
obtained after all closures have been performed (there are $i$ closures, as each closure is associated
with one of the $i$ non-root black vertices of $\Tb$). Note that each bounded
face $f$ of $F$ has been ``closed'' by matching a red half-edge going out of a black 
vertex $b$ with a red half-edge going into a white vertex $w$. The vertex $b$
is called the \emph{left-vertex} of~$f$. 

Let us now describe how to complete $F$ into a separating decomposition on a 
quadrangulation.  
Add an isolated vertex $t$ in the outer face of $F$.
Taking advantage of Invariant $(I)$, it is easy to complete suitably each red stem $h$
going out of a white vertex:
\begin{itemize}
\item
if $h$ is in a bounded face $f$ of  $F$ we complete $h$ into an edge 
connected to the left-vertex of $f$; completing all the half-edges
inside the face $f$ splits $f$ into quadrangular faces, as shown in 
Figure~\ref{fig:face_a}. 
\item
if $h$ is in the outer face of $F$ we complete $h$ into an edge connected to the 
vertex $t$; completing all such half-edges splits 
the outer face of $F$ into quadrangular faces all incident to $t$, and $t$ is incident to red
incoming edges only, see Figure~\ref{fig:face_b}.
\end{itemize}
The planar map we obtain is thus a quadrangulation.
In addition it is easy to check that the orientations and colors of the edges satisfy the
local conditions of a separating decomposition.
Indeed, the local conditions are satisfied in $F$. Afterwards
 the (black) left-vertex of each bounded face of $F$ receives new incoming red edges
in cw order after the red outgoing edge, and the vertex $t$ receives 
red incoming edges only. Hence the local conditions 
 remain satisfied after inserting the last red half-edges.

To sum up, we have described a mapping $\Psi$ from non-intersecting triples
of paths of type $(i,j)$   to separating decompositions with $i+2$ black vertices and 
$j+2$ white vertices.
It is easy to check step by step that the mapping $\Phi$ described in Section~\ref{sec:encode} 
and the mapping $\Psi$ 
are mutually inverse. %This yields the bijective result stated in Theorem. 
Together with Proposition~\ref{theo:bi}, this yields our main bijective result announced in Theorem~\ref{theo:bijection}.

\section{Specialization into a bijection for Schnyder woods}\label{sec:schnyder}
A \emph{triangulation} is a planar map with no loop nor multiple edge such that each face is triangular.
Given a triangulation $T$, let $s,t,u$ be its outer vertices in cw order. 
A Schnyder wood on $T$ is an orientation and coloration---in blue, red, or green---of the inner edges of $T$ such that the following local conditions are satisfied (in the figures, blue edges are solid, red edges are dashed, and green edges are dotted):
\begin{itemize}
\item
Each inner vertex $v$ of $T$ has exactly one outgoing edge in each color. The edges leaving  $v$ in color
 blue, green, and red, occur in cw order around $v$.
In addition, the incoming edges of one color appear between the outgoing edges of the two other colors,
see Figure~\ref{fig:SchnyderSep}(a).
\item
All the  inner edges incident to the outer vertices  are incoming, and such edges are colored 
blue, green, or red, whether the outer vertex is $s$, $t$, or $u$, respectively.
\end{itemize}

Definition, properties, and applications of 
Schnyder woods are given in Felsner's monograph~\cite[Chapter 2]{FeBook}.
Among the many properties of Schnyder woods, it is well known 
that the subgraphs of $T$
in each color are trees that span all the inner vertices and one outer vertex (each of the 3 outer vertices
is the root of one of the trees).

We show here that  Schnyder woods are in bijection with specific separating decompositions, and that
 such separating decompositions have one of the 3 encoding paths that is redundant, and  the two 
other ones are Dyck paths. Afterward we show that this bijection is
exactly the one recently described by Bernardi and Bonichon in~\cite{BeBo07}
(which itself reformulates Bonichon's original construction~\cite{B02}).

Starting from a Schnyder wood $S$ with $n$ inner vertices, 
we construct a separating decomposition $D=\alpha(S)$
as follows, see Figure~\ref{fig:SchnyderSep}:
\begin{itemize}
\item
Split each inner vertex $v$ of $T$ into a white vertex $w$ and a black vertex $b$
that are connected by a blue edge going from $b$ to $w$. In addition $w$ receives
the outgoing green edge, the outgoing blue edge and the incoming red edges of $v$, and $b$ 
receives the outgoing red edge, the incoming blue edges, and the incoming green edges of $v$.
\item
Add a white vertex in the middle of the edge $(s,t)$, and change the color of $u$ from black to white.
\item
Recolor the green edges into red edges.
\item
Color red the two outer edges incident to $t$ and orient these edges toward $t$. Color
blue the two outer edges incident to $s$ and orient these edges toward $s$.
\end{itemize}

\begin{figure}
\centering
\includegraphics[width=15cm]{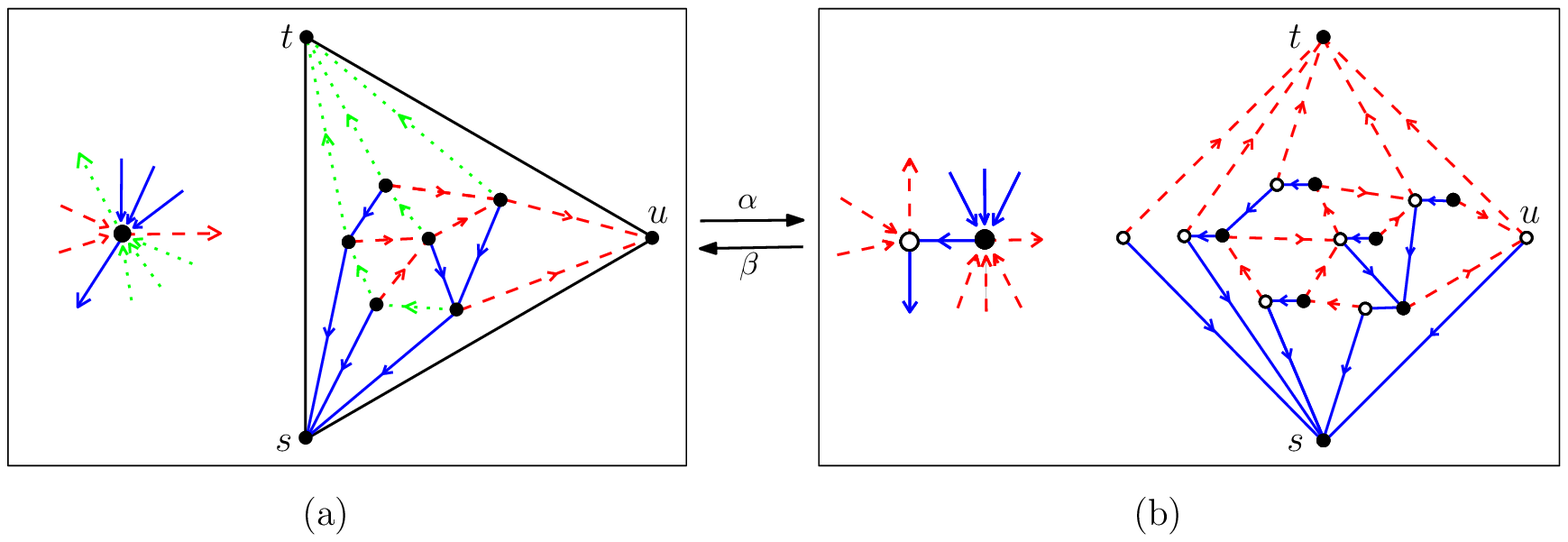}
\caption{From a Schnyder wood to a contractible separating decomposition.}
\label{fig:SchnyderSep}
\end{figure}

Clearly we obtain from this construction a bipartite planar map $Q$ with no multiple edge.
The map $Q$ has 
a quadrangular outer face, $2n$ inner vertices,
and $4n$ inner edges (the $3n$ inner edges of the original triangulation plus the $n$ new edges),
hence $Q$ has to be a maximal bipartite planar map, \ie $Q$ is a quadrangulation.
In addition, it is easily checked that $Q$ is endowed with a separating decomposition $D=\alpha(S)$ via the construction,
as shown in Figure~\ref{fig:SchnyderSep}. 
A separating decomposition 
is called \emph{contractible} if each inner white vertex has blue indegree equal to $1$  and the two outer white 
vertices have blue indegree $0$. Clearly $D=\alpha(S)$ is contractible, see Figure~\ref{fig:SchnyderSep}(b).

Conversely, starting from a contractible separating decomposition $D$, we construct the associated Schnyder
wood $S=\beta(D)$ as follows:
\begin{itemize}
\item
recolor the red edges of $D$ going out of white vertices into green edges.
\item
contract the blue edges going from a black to a white vertex.
\item
remove the colors and directions of the outer edges of $D$; contract into a single edge the path of 
length $2$ going from $s$ to $t$ with the outer face on its left.
\end{itemize}

Clearly, the local conditions of Schnyder woods are satisfies by $S$. Hence, proving that $S$
is a Schnyder wood comes down to proving that the planar map we obtain is a triangulation.
In fact, it is enough to show that all faces are triangular (it is well known that a map with all
faces of degree 3 and endowed with a Schnyder wood has no loop nor multiple edges),
which clearly relies on the following lemma.

\begin{lemma}
Take a contractible separating decomposition $D$ and remove the path of length 2 going
from $s$ to $t$ with the outer face on its left (which yields a separating decomposition with 
one inner face less).
Then around each inner face there is exactly one blue edge
going from a black vertex to a white vertex.
\end{lemma}
\begin{proof}
Let $O$ be the plane bipolar orientation associated to $D$. Observe that $s$ and $t$ are adjacent
in $O$, the edge $(s,t)$ having the outer face of $O$ on its left. To each inner face $f$ of $O$ 
corresponds the unique edge $e$ of $O$ that is in the interior of $f$. For each edge $e$ of $O$,
except for $(s,t)$, let $\ell_e$ ($r_e$) be the face of $O$ on the left (right, \resp) of $e$,
and let $w_{\ell}$ ($w_r$, \resp) be the corresponding white vertex on $D$. 
Notice that the inner face of $D$ associated with $e$ is the face $f$ incident to the extremities of $e$
and to the white vertices $w_{\ell}, w_r$.
As $w_{\ell}$ has
blue indegree $1$, $\ell_e$ has two edges on its right side. Hence, one extremity $v$ of $e$
is extremal for $\ell_e$, and the other extremity $v'$ of $e$ is in the middle of the right side of $\ell_e$.
Hence the edge $(v',w_{\ell})$, which is on the contour of $f$, is a blue edge with a black origin.
In addition, the edge $(w_{\ell},v)$ goes into $v$ (as $v$ is extremal for $\ell_e$), and each of the other two edges of $f$ is  either red or is blue with $w_r$ as origin, by the rules to translate a plane bipolar orientation into
a separating decomposition.  Hence any inner face of $D$, except the one corresponding to $(s,t)$,
has on its contour exactly one blue edge with a black origin.
\end{proof}

Clearly the mappings $\alpha$ and $\beta$ are mutually inverse, so that we obtain the following
result (which to our knowledge is new):

\begin{proposition}\label{prop:schny_dec}
Schnyder woods with $n$ inner vertices are in bijection with contractible separating decompositions
with $n$ black inner vertices.
\end{proposition}

\begin{figure}
\centering
\includegraphics[width=15cm]{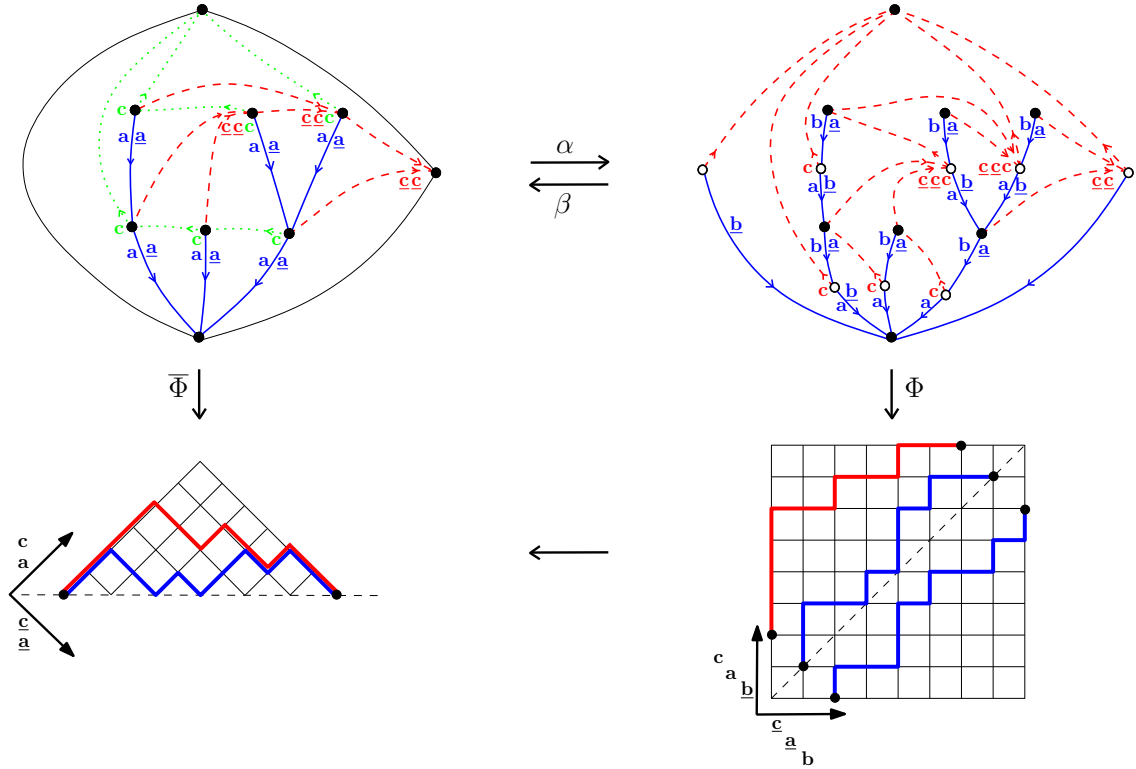}
\caption{Encoding a Schnyder wood by two non-crossing Dyck paths via the associated 
contractible separating decomposition.}
\label{fig:encodeSchnyderSep}
\end{figure}

%Proposition~\ref{prop:schny_dec} ensures that  Schnyder woods with $n$ inner vertices are in bijection
%with the non-intersecting triples of paths  that encode 
 %contractible separating decomposition with $n$ inner black vertices. 
 
 Let us now describe the non-intersecting triples of paths associated
 with  contractible separating
 decompositions. 
 Let  $D$ be  
  a contractible separating decomposition with $2n$ inner vertices, 
  and let $(P_a',P_b',P_c')=\Phi(D)$ be 
  the associated non-intersecting triple of paths, which has type $(n,n)$. 
  Let $\Tb$ be the blue tree of $D$. Observe that $\Tb$ has one 1-leg on the left and on the right side and all  the other white vertices have exactly one child.
  Let $T$ be the tree obtained from $\Tb$ by deleting the 1-legs on each side and by 
  merging each white vertex with its unique black child. 
  Then it is easily checked that $P_a'$ is the 
 Dyck path encoding $T$. In addition, $P_b'$ is
 redundant: it is obtained as the mirror of  $P_a'$ \wrt the diagonal $x=y$,  shifted
 one step to the right, and with the last (up) step moved so as to prepend the path,
 see Figure~\ref{fig:encodeSchnyderSep} (right part).  
Finally the path $P_c'$ is also a Dyck path, since it does 
 not intersect $P_a'$ and its respective endpoints are one step up-left of the 
 corresponding endpoints of $P_a'$.  
  To have a more classical representation, one rotates cw by 45 degrees the two paths $P_c'$ and $P_a'$ and shifts them to have the same starting point (and same endpoint),
  see Figure~\ref{fig:encodeSchnyderSep} (lower part). After doing this, the pair 
  $(P_a',P_c')$ is 
  a non-crossing pair of  Dyck paths (each of length $2n$) that is enough to encode the separating decomposition.
  
  Conversely, starting from a pair $(P_a',P_c')$ of non-crossing Dyck paths, we rotate the two paths
  ccw by 45 degrees and shift $P_c'$ one step up-left, so that $P_c'$ now does not intersect $P_a'$. 
  Then we construct the path $P_b'$ as the 
  mirror of $P_a'$ according to the diagonal $x=y$, with  the last step moved to the start of the path, and we place $P_b'$ 
  so as to have its starting point  one step bottom-right of the starting point of $P_a'$. As $P_a'$
  is a Dyck word (\ie stays weakly above the diagonal $x=y$), the path $P_b'$ does not intersect $P_a'$. Furthermore it is easily checked that the blue tree $\Tb$ of the separating decomposition  $D=\Psi(P_b',P_a',P_c')$ has one 1-leg on each side and all other white vertices have one child in $\Tb$.
(Proof: by definition of $\Psi$, the Dyck path $P$ for $\Tb$ is obtained as a shuffle at even and odd positions of the path $P_a:=\uparrow P_a'\uparrow$ and of the path
$P_b:=P_b'\uparrow\uparrow$. By construction of $P_b'$ from $P_a'$, it is
easily checked that there is a $\wedge$ at the beginning---starting at position 0---and at the end of $P$ and that all the other peaks and valleys of $P$ start at odd position, hence the corresponding leaves and forks of $T$ are at black vertices only.)

To conclude, we have proved that contractible separating decompositions with $n$ 
inner vertices
are encoded (via the bijection $\Phi$) by non-crossing pairs of Dyck paths each having
$2n$ steps. Given 
Proposition~\ref{prop:schny_dec}, we recover Bonichon's result~\cite{B02}: 

\begin{theorem}\label{theo:sch}
Schnyder woods on triangulations with $n$ inner vertices are in bijection with non-crossing pairs of Dyck paths that have both $2n$ steps.
\end{theorem}

As shown in Figure~\ref{fig:encodeSchnyderSep}, 
the bijection can be formulated as a mapping $\overline{\Phi}$
operating directly on the Schnyder wood $S$. Indeed, let $D=\alpha(S)$.
The blue tree $T$ of $S$ is equal to the tree $\Tb$ of $D$ where the 1-legs on each
side are deleted and where each white node is merged with its unique black child.
Hence the path $P_a'$ (the lower Dyck path) 
associated with $D$ is the Dyck path encoding the blue 
tree $T$ of $S$. 
And the upper Dyck path $P_c'$
can be read directly on $S$: $P_c'$ is obtained by walking cw around
$T$, 
drawing an up-step (down-step) each time an outgoing green edge (incoming red edge, \resp) is crossed, and completing the end of the path by down-steps. 
This mapping is  exactly  the bijection that has been recently described by Bernardi and Bonichon~\cite{BeBo07} for counting Schnyder woods (and more generally for counting some intervals of Dyck paths), 
which itself is a reformulation of Bonichon's original construction~\cite{B02}.

\bibliographystyle{plain}
\bibliography{mabiblio}  

\begin{thebibliography}{10}

\bibitem{baxter}
R.~J. Baxter.
\newblock Dichromatic polynomials and potts models summed over rooted maps.
\newblock {\em Annals of Combinatorics}, 5:17, 2001.

\bibitem{BeBo07}
O.~Bernardi and N.~Bonichon.
\newblock Intervals in catalan lattices and realizers of triangulations.
\newblock {\em Journal of Combinatorial Theory. Series A}, 116(1):55--75, 2009.

\bibitem{Bi}
T.~Biedl and F.~J. Brandenburg.
\newblock Partitions of graphs into trees.
\newblock In {\em Proceedings of Graph Drawing'06 (Karlsruhe)}, volume 4372 of
  {\em LNCS}, pages 430--439, 2007.

\bibitem{B02}
N.~Bonichon.
\newblock A bijection between realizers of maximal plane graphs and pairs of
  non-crossing dyck paths.
\newblock {\em Discrete Math.}, 298(1-3):104--114, 2005.

\bibitem{BoBoFu08}
N.~Bonichon, M.~Bousquet-M\'elou, and \'E. Fusy.
\newblock Baxter permutations and plane bipolar orientations.
\newblock arXiv:0805.4180, 2008.

\bibitem{bousquet-melou-four}
M.~Bousquet-M\'elou.
\newblock Four classes of pattern-avoiding permutations under one roof:
  Generating trees with two labels.
\newblock {\em Elect. J. Comb.}, 2002.

\bibitem{BT64}
W.~G. Brown and W.T. Tutte.
\newblock On the enumeration of rooted nonseparable maps.
\newblock {\em Canad. J. Math.}, 16:572--577, 1964.

\bibitem{Chiba}
N.~Chiba, T.~Nishizeki, S.~Abe, and T.~Ozawa.
\newblock A linear algorithm for embedding planar graphs using pq-trees.
\newblock {\em J. Comput. Syst. Sci.}, 30(1):54--76, 1985.

\bibitem{DeOss}
H.~De~Fraysseix, P.~Ossona~de Mendez, and P.~Rosenstiehl.
\newblock Bipolar orientations revisited.
\newblock {\em Discrete Appl. Math.}, 56(2-3):157--179, 1995.

\bibitem{DuGu98}
S.~Dulucq and O.~Guibert.
\newblock Baxter permutations.
\newblock {\em Discr. Math.}, 180:143--156, 1998.

\bibitem{FeBook}
S.~Felsner.
\newblock {\em Geometric graphs and arrangements}.
\newblock Vieweg Verlag, 2004.

\bibitem{FeFuNoOr07}
S.~Felsner, \'E. Fusy, M.~Noy, and D.~Orden.
\newblock Bijections for baxter families and related objects.
\newblock arXiv:0803.1546, 2008.

\bibitem{Fu06}
\'E. Fusy.
\newblock Straight-line drawing of quadrangulations.
\newblock In {\em Proceedings of Graph Drawing'06}, volume 4372 of {\em LNCS},
  pages 234--239. Springer, 2006.

\bibitem{GeVi1}
I.~Gessel and X.~Viennot.
\newblock Binomial determinants, paths, and hook formulae.
\newblock {\em Adv. Math}, 58:300--321, 1985.

\bibitem{GeVi2}
I.~Gessel and X.~Viennot.
\newblock Determinants, paths, and plane partitions, 1989.
\newblock Preprint.

\bibitem{Hu}
C.~Huemer and S.~Kappes.
\newblock A binary labelling for plane laman graphs and quadrangulations.
\newblock In {\em Proceedings of EWCG, Delphi}, pages 83--86, 2006.

\bibitem{Kant}
G.~Kant and X.~He.
\newblock Regular edge labeling of $4$-connected plane graphs and its
  applications in graph drawing problems.
\newblock {\em Theoretical Computer Science}, 172(1-2):175--193, 1997.

\bibitem{Le66}
A.~Lempel, S.~Even, and I.~Cederbaum.
\newblock An algorithm for planarity testing of graphs.
\newblock In {\em Theory of Graphs, Int. Symp (New York)}, pages 215--232,
  1967.

\bibitem{Oss}
P.~Ossona~de Mendez.
\newblock {\em Orientations bipolaires}.
\newblock PhD thesis, Ecole des Hautes Etudes en Sciences Sociales, Paris,
  1994.

\bibitem{TaTo}
R.~Tamassia and I.~G. Tollis.
\newblock A unified approach to visibility representations of planar graphs.
\newblock {\em Discrete Comput. Geom.}, 1(4):321--341, 1986.

\bibitem{TaTo2}
R.~Tamassia and I.~G. Tollis.
\newblock Planar grid embedding in linear time.
\newblock {\em IEEE Trans. on Circuits and Systems}, CAS-36(9):1230--1234,
  1989.

\end{thebibliography}
\end{document}